\documentclass{amsart}
\usepackage[all]{xy}\usepackage{amssymb,latexsym}

\newtheorem{theorem}{Theorem}
\newcommand{\bt}{\begin{theorem}}
\newcommand{\et}{\end{theorem}}
\newtheorem{lemma}{Lemma}

\newcommand{\Rn}{\ensuremath{ \mathbf{R}^n }}

\newcommand{\beq}{\begin{equation}}
\newcommand{\eeq}{\end{equation}}
\newcommand{\benum}{\begin{enumerate}}
\newcommand{\eenum}{\end{enumerate}}

\newtheorem{problem}{Problem}
\newcommand{\bprob}{\begin{problem}}
\newcommand{\eprob}{\end{problem}}

\newcommand{\R}{\ensuremath{\mathbf R}}
\newcommand{\C}{\ensuremath{\mathbf C}}

\title[A generalised Hermite-Sylvester theorem]
{Real-rooted polynomials and a generalised Hermite-Sylvester theorem}
\author{Melvyn B. Nathanson}
\address{Department of Mathematics\\Lehman College (CUNY)\\Bronx, NY 10468 USA}
\email{melvyn.nathanson@lehman.cuny.edu}


\thanks{Supported in part by a grant from the PSC-CUNY Research Award Program.}

\date{\today}

\begin{document}

\maketitle

\begin{abstract}
A polynomial is \emph{real-rooted} if all of its roots are real.  
For every polynomial $f(t) \in {\mathbf R}[t]$, 
the Hermite-Sylvester theorem associates a quadratic form $\Phi_2$ such that $f(t)$  is real-rooted 
if and only if $\Phi_2$ is positive semidefinite.   In this note, for every positive integer $m$, an $m$-adic 
form $\Phi_{2m}$ is constructed such that $f(t)$  is real-rooted 
if and only if $\Phi_{2m}$ is positive semidefinite for some  positive integer $m$ 
if and only if $\Phi_{2m}(x_1,\ldots, x_n)$ is positive semidefinite for every  positive integer $m$.  
\end{abstract}

A polynomial is \emph{real-rooted} if all of its roots are real. 
A real-valued function $\Phi(x_1,\ldots, x_n)$ is \emph{positive semidefinite}\index{positive semidefinite}  
if  $\Phi(x_1,\ldots, x_n) \geq 0$ for all vectors $(x_1,\ldots, x_n) \in \Rn$.  
An \emph{$n$-ary $m$-adic form} is a polynomial in $n$ variables that is homogeneous 
of degree $m$.  

Let $f(t)$ be a polynomial of degree $n \geq 1$ with real coefficients, 
and  let $\lambda_1,\ldots, \lambda_n$ be the (not necessarily distinct) roots of $f(t)$.   
Hermite and Sylvester~\cite{nath2021-188,vond16}  proved that the polynomial 
$f(t)$ is real-rooted 
if and only if the quadratic form in $n$ variables
\[
\Phi_{2}(x_1,\ldots, x_n) 
= \sum_{j_1=1}^n  \sum_{j_2=1}^n  \sum_{\ell=1}^n \lambda_{\ell}^{j_1+j_2-2} x_{j_1} x_{j_2} 
\]
is positive semidefinite. 
This note generalises the Hermite-Sylvester theorem. 
We define, for every positive integer $m$, the following $n$-ary $m$-adic form: 
 \beq                                       \label{HSN:Phi}    
\Phi_{m}(x_1,\ldots, x_n)      
 = \sum_{j_1=1}^n  \sum_{j_2=1}^n \cdots  \sum_{j_{m} =1}^n 
\sum_{\ell=1}^n \lambda_{\ell}^{ \sum_{t=1}^{m} ( j_t -1) } x_{j_1} x_{j_2} \cdots x_{j_{m}}.       
\eeq
We shall prove that the the polynomial $f(t)$  is real-rooted if, and only if, 
$\Phi_{m}(x_1,\ldots, x_n)$ is positive semidefinite for all positive even integers $m$.

\begin{lemma}                           \label{HSN:lemma:Phi}
For all positive integers $m$, 
\beq                           \label{HSN:f}
\Phi_{m}(x_1,\ldots, x_n)  = \sum_{\ell=1}^n p(\lambda_{\ell})^{m}                      
\eeq
where
\beq                           \label{HSN:p}
p(t) = \sum_{j=1}^n x_j t^{j-1}.
\eeq
\end{lemma}

\begin{proof}
We have 
\begin{align*}                              
\Phi_{m}(x_1,\ldots, x_n)        
& = \sum_{j_1=1}^n  \sum_{j_2=1}^n \cdots  \sum_{j_{m} =1}^n 
\sum_{\ell=1}^n \lambda_{\ell}^{ \sum_{t=1}^{m} ( j_t -1) } x_{j_1} x_{j_2} \cdots x_{j_{m}}      \\
& =  \sum_{j_1=1}^n  \sum_{j_2=1}^n \cdots  \sum_{j_{m} =1}^n 
\sum_{\ell=1}^n  \prod_{t=1}^{m}  x_{j_t}  \lambda_{\ell}^{  j_t-1}      \\
& = \sum_{\ell=1}^n  \sum_{j_1=1}^n  \sum_{j_2=1}^n \cdots  \sum_{j_{m} =1}^n 
 \prod_{t=1}^{m}  x_{j_t}  \lambda_{\ell}^{  j_t-1}      \\
& = \sum_{\ell=1}^n \left( \sum_{j=1}^n x_j \lambda_{\ell}^{j-1} \right)^{m}          \\
& = \sum_{\ell=1}^n p(\lambda_{\ell})^{m}                     
\end{align*}
where $p(t)$ is the polynomial defined by~\eqref{HSN:p}. 
\end{proof}

\begin{lemma}                   \label{perm:lemma:Lagrange-1}
Let $\Lambda  = \{\lambda_1,\lambda_2,\lambda_3,\ldots, \lambda_r\}$ 
be a nonempty finite set of $r$ complex numbers that 
is closed under complex conjugation.  
Suppose that $\lambda_1 \in \Lambda$ is not real, and that $\lambda_2 = \overline{\lambda_1} \in \Lambda$.  
Let $\omega \in \C$ with complex comjugate $\overline{\omega}$.  
There exists a polynomial $p(t) \in {\mathbf R}[t]$ of degree at most $r - 1$ 
such that 
\begin{equation}                    \label{perm:Lagrange-1}
p(\lambda_1) = \omega \qquad p(\lambda_2) = \overline{\omega}
\end{equation}
and
\begin{equation}                    \label{perm:Lagrange-2}
p(\lambda_j) = 0 \qquad \text{for all $j \in \{3,4,\ldots, r\}$.}
\end{equation}
\end{lemma}

\begin{proof}
This was proved in Nathanson~\cite{nath2021-188} in the special case $\omega = i$, 
but the proof applies unchanged for every complex number $\omega$.  
\end{proof}

\bt            \label{HSN:theorem:generalised}
The following are equivalent:
\benum
\item[(a)]
The polynomial $f(t)$ has real roots.
\item[(b)]
The $n$-ary $m$-adic form $\Phi_{m}(x_1,\ldots, x_n)$ is positive semidefinite 
for all positive even integers $m$.
\item[(c)] 
The $n$-ary $m$-adic form $\Phi_{m}(x_1,\ldots, x_n)$ is positive semidefinite 
for some positive even integer $m$.
\eenum
\et

\begin{proof}
Let $m$ be a  positive even integer.  
Let $\Phi_{m}$ be the $n$-ary $m$-adic form defined by~\eqref{HSN:Phi}, 
and let  $p(t)$ be the polynomial defined by~\eqref{HSN:p}.  By Lemma~\ref{HSN:lemma:Phi}, 
\[                    
\Phi_{m}(x_1,\ldots, x_n) = \sum_{\ell=1}^n p(\lambda_{\ell})^{m}.         
\]
For all vectors $(x_1,\ldots, x_n) \in \Rn$, the polynomial $p(t)$ has real coefficients. 
Because $m$ is even, if the root $\lambda_{\ell}$ of the polynomial $f(t)$ is real, then $p(\lambda_{\ell})^{m} \geq 0$.  
It follows that if  every root of the polynomial $f(t)$ is real,
then the form $\Phi_{m}$ is positive semi-definite for all positive  even integers $m$.   
Thus, (a) implies (b).   Also, (b) implies (c).

Suppose that $f(t)$ has a non-real root.  
The non-real roots of a polynomial with real coefficients occur in complex conjugate pairs.  
Renumbering the roots $\lambda_1,\ldots, \lambda_n$ of $f(t)$, 
we can assume that 
$\lambda_1,\ldots, \lambda_r$ are the distinct roots of $f(t)$, that 
$\lambda_1$ is non-real, and that $\lambda_2 = \overline{\lambda_1}$.   
 If $\mu$ is the multiplicity of the root $\lambda_1$, 
then $\mu$ is  also the multiplicity of the root $\lambda_2$. 

Let $m$ be a positive even integer.  
The method of Lemma 2 of Nathanson~\cite{nath2021-188}  
produces a polynomial  $p(t) = \sum_{j=1}^n x_j t^{j-1} \in \R[t]$  of degree at most $r-1 \leq n-1$ such that 
\[
p(\lambda_1) = e^{\pi i/m},  \qquad p(\lambda_2) = e^{-\pi i/m}, 
\]
and
\[
p(\lambda_{\ell}) = 0 \qquad\text{for $\ell = 3,4,\ldots, r$.}
\]
The polynomial $p(t)$ and the vector $(x_1,\ldots, x_n) \in \Rn$ depend on $m$.  
We have 
\[
p(\lambda_1)^{m} = p(\lambda_2)^{m} = -1 
\]
and 
\begin{align*}
\Phi_{m}(x_1,\ldots, x_n) 
&  = \sum_{\ell=1}^n p(\lambda_{\ell})^{m}  \\
& = \mu p(\lambda_1)^{m} + \mu p(\lambda_2)^{m} \\
&  = -2\mu  < 0.
\end{align*}
Thus, if $f(t)$ has a non-real root, then the form $\Phi_{m}$ is not positive semidefinite.  
Equivalently, if  $\Phi_{m}$ is positive semidefinite for some  positive even integer $m$, 
then the polynomial $f(t)$ is real-rooted.  
Thus,  (c) implies (a).  
This completes the proof.  
\end{proof}

Here is an example.  
Let $f(t)$ be a quadratic polynomial with roots $\lambda_1$ and $\lambda_2$. 
We have $p(t) = x_1 + x_2t$  and, for every positive integer $m$,  
the   binary $m$-adic form 
\begin{align*}
\Phi_{m}(x_1,x_2) 
& = p(\lambda_1)^m + p(\lambda_2)^m  = (x_1+ \lambda_1x_2)^{m} + (x_1 + \lambda_2 x_2)^{m} \\
 & = \sum_{k=0}^{m}   \binom{m}{k} \left(  \lambda_1^k  + \lambda_2^k \right) x_1^{m-k}x_2^k.  
\end{align*}
For the polynomial  $f(t) = t^2-1$ with real roots $\lambda_1 = 1$ and $\lambda_2 = -1$, 
we have 
\[
 \lambda_1^k  + \lambda_2^k = \begin{cases}
 2 & \text{if $k \equiv 0 \pmod{2}$} \\
0 & \text{if $k \equiv 1 \pmod{2}$} \\
 \end{cases}
\]
and 
\[
\Phi_{m}(x_1,x_2) 
 = 2 \sum_{\substack{k=0 \\ k \equiv 0 \pmod{2}}}^{m} \binom{m}{k} x_1^{m-k}x_2^{k}.
\]
Thus, 
\begin{align*}
\Phi_1(x_1,x_2) &  =  2x_1   \\
\Phi_2(x_1,x_2) & = 2\left(  x_1^2+ x_2^2  \right) \\
\Phi_3(x_1,x_2) &  =  2x_1 \left(  x_1^2+ 3 x_2^2  \right) \\
\Phi_4(x_1,x_2)
& = 2 \left(   x_1^4 + 6 x_1^2  x_2^2 + x_2^4 \right)  \\
\Phi_5(x_1,x_2) &  =  2x_1 \left( x_1^4+ 10x_1^2 x_2^2+ 5x_2^4  \right) \\
\Phi_6(x_1,x_2) & = 2 \left(  x_1^6 + 15 x_1^4  x_2^2  + 15 x_1^2  x_2^4  + x_2^6 \right).
%\Phi_7(x_1,x_2) &  =  2x_1 \left( x_1^6+ 21x_1^4 x_2^2+ 35 x_1^2x_2^4  +7x_2^6\right).
\end{align*}

For the polynomial  $t^2+t+1$ with nonreal roots $\lambda_1 = -1/2 + i\sqrt{3}/2$ 
and $\lambda_2 = -1/2 - i\sqrt{3}/2$, we have \[
 \lambda_1^k  + \lambda_2^k = \begin{cases}
 2 & \text{if $k \equiv 0 \pmod{3}$} \\
-1 & \text{if $k \equiv \pm1 \pmod{3}$} \\
 \end{cases}
\]
and 
\[
\Phi_{m}(x_1,x_2) 
 = 2 \sum_{\substack{k=0 \\ k \equiv 0 \pmod{3}}}^{m} \binom{m}{k} x_1^{m-k}x_2^{k} 
 -  \sum_{\substack{k=0 \\ k \equiv \pm 1\pmod{3}}}^{m} \binom{m}{k} x_1^{m-k}x_2^{k}.
\]
Thus, 
\begin{align*} 
\Phi_1(x_1,x_2) &  = 2x_1-x_2  \\
\Phi_2(x_1,x_2) &  = 2 x_1^2 - 2x_1x_2  - x_2^2 \\ 
\Phi_3(x_1,x_2) &  =  2x_1^3-3x_1^2x_2-3x_1x_2^2+2x_2^3 \\
\Phi_4(x_1,x_2) &  = 2x_1^4-4x_1^3x_2-6x_1^2x_2^2+8x_1x_2^3-x_2^4\\ 
\Phi_5(x_1,x_2) &  = 2x_1^5-5x_1^4x_2-10x_1^3x_2^2+20x_1^2x_2^3-5x_1x_2^4-x_2^5 \\
\Phi_6(x_1,x_2) & = 2x_1^6-6x_1^5x_2-15x_1^4x_2^2+40x_1^3x_2^3-15x_1^2x_2^4-6x_1x_2^5+2x_2^6.
%\Phi_7(x_1,x_2) & = 2x^7 - 7x^6y - 21x^5y^2 + 70x^4y^3 - 35x^3y^4 - 21x^2y^5 + 14xy^6 - y^7.
\end{align*}

Theorem~\ref{HSN:theorem:generalised} suggests the following problems.

\bprob
Let $f(t) \in \R[t]$ be a polynomial of degree $n$, and let $(\Phi_{m})_{m=1}^{\infty}$ 
be the sequence of $n$-ary $m$-adic forms defined by~\eqref{HSN:Phi}.
What does the subsequence $(\Phi_{2m+1})_{m=0}^{\infty}$ of forms  
of odd degree tell us about the polynomial $f(t)$?
\eprob

\bprob
Does the sequence of polynomials $(\Phi_m)_{m=1}^{\infty}$ constructed by~\eqref{HSN:Phi} 
satisfy a nice recursion relation? 
\eprob

\bprob
Let $f(t)$ be a quadratic polynomial. 
Is there a binary cubic form $\Psi_{3}(x_1,x_2)$ 
associated with  $f(t)$ such that $\Psi_{3}$ determines if $f(t)$ has real roots? 
\eprob

\bprob
More generally, given a polynomial $f(t)$ of degree $n$, 
does there exist a homogeneous $n$-ary form  $\Psi_{2m+1}$ of odd degree $2m+1$ 
associated with the polynomial such that the form $\Psi_{2m+1}$ determines if the polynomial has real roots? 
\eprob

\def\cprime{$'$} \def\cprime{$'$} \def\cprime{$'$}
\providecommand{\bysame}{\leavevmode\hbox to3em{\hrulefill}\thinspace}
\providecommand{\MR}{\relax\ifhmode\unskip\space\fi MR }
\providecommand{\MRhref}[2]{%
  \href{http://www.ams.org/mathscinet-getitem?mr=#1}{#2}
}
\providecommand{\href}[2]{#2}

\end{document}